\documentclass [12pt,a4paper]{article}

\usepackage{a4}
\usepackage{exscale}
\usepackage[centertags]{amsmath}
\usepackage{amssymb}
\usepackage{amsthm}
\usepackage{ulem}
\usepackage{isolatin1}
\usepackage{color}
\usepackage{mathtools}
\usepackage{enumitem}

\def\parmod{\parskip=2pt plus1pt minus1pt}

\renewcommand{\sectionmark}[1]
                    {\markboth{Kapitel \thesection\ #1}{}}
\renewcommand{\sectionmark}[1]
                 {\markright{} }

\numberwithin{equation}{section}

\newtheorem{theorem}{Theorem}

\newtheorem{lemma} {Lemma}
\newtheorem{corollary}[theorem]{Corollary}

\newtheorem{remark}[theorem]{Remark}

\renewenvironment{proof}
       {\pagebreak[2] \vspace{-1pt}{\bf Proof.}  }
      {\hfill $\blacksquare$ \vspace{2pt}}

\newcommand{\Aa}{\mathcal{A}}

\newcommand{\nat}{{\rm I\! N}}

\renewcommand{\r}{\right}
\newcommand{\gl}{\left\{}
\newcommand{\gr}{\right\}}
\newcommand{\kl}{\left(}
\newcommand{\kr}{\right)}

\newcommand{\limn}{\lim_{n\to\infty}}
\newcommand{\limsupn}{\limsup_{n\to\infty}}

\newcommand{\ti}{\widetilde}

\newcommand{\mi}{\setminus}
\newcommand{\abb}{\longrightarrow}

\renewcommand{\rho}{\varrho}
\renewcommand{\phi}{\varphi}
\renewcommand{\epsilon}{\varepsilon}

\def\Log{\operatorname{log}}

\newcommand{\cupdot}{\mathbin{\dot\cup}}
\newcommand{\bigcupdot}{\mathop{\dot\bigcup}}

\newenvironment{rem}
        {\pagebreak[2] \begin{remark} \quad \parmod \rm}
        {\end{remark}}
\newenvironment{lem}
        {\pagebreak[2] \begin{lemma} \quad \parmod \sl}
        {\end{lemma}}
\newenvironment{thm}
        {\pagebreak[2] \begin{theorem} \quad \parmod \sl}
        {\end{theorem}}

\newcommand{\beq}{\begin{equation}}
\newcommand{\eeq}{\end{equation}}
\newcommand{\beqar}{\begin{eqnarray}}
\newcommand{\eeqar}{\end{eqnarray}}
\newcommand{\beqaro}{\begin{eqnarray*}}
\newcommand{\eeqaro}{\end{eqnarray*}}
\newcommand{\bsat}{\begin{thm}}
\newcommand{\esat}{\end{thm}}
\newcommand{\bsatorig}{\begin{satOrig}}
\newcommand{\esatorig}{\end{satOrig}}
\newcommand{\blem}{\begin{lem}}
\newcommand{\elem}{\end{lem}}
\newcommand{\bkor}{\begin{corollary}}
\newcommand{\ekor}{\end{corollary}}
\newcommand{\bbew}{\begin{proofof}}
\newcommand{\ebew}{\end{proofof}}
\newcommand{\brem}{\begin{rem}}
\newcommand{\erem}{\end{rem}}

\begin{document}

\title{Estimates for probabilities of independent events and infinite series}
\author{J\"urgen Grahl and Shahar Nevo}
\date{\today}

\maketitle

\begin{abstract}
This paper deals with (finite or infinite) sequences of arbitrary
independent events in some probability space. We find sharp lower
bounds for the probability of a union of such events when the sum of
their probabilities is given. The results have parallel meanings in
terms of infinite series. 
\end{abstract}



\section{Introduction}

This paper deals with (finite or infinite) sequences of arbitrary
independent events in some probability space $(\Omega,\Aa,P)$. In particular,
we discuss the connection between the sum of the probabilities of these
events and the probability of their union. Naturally, the results can
be formulated both in the ``language'' of probability and the
``language'' of calculus of non-negative series. 

This paper is written in an expository and to some extent educational
style. Part of the results (in particular in sections 2 and 3) are
basically known, one of them being more or less equivalent to the
Borel-Cantelli lemma. We hope that our approach, emphasizing the
connections to calculus, will be of interest in itself.

% trying to give an approach to some topics from probability
%theory that slightly deviates from the usual one. 

In Section 2 we start with a lemma/construction that shows that for
each sequence $\{x_n\}_{n=1}^\infty$ of real numbers $x_n\in[0;1)$
there is a sequence of independent events $\{A_n\}_{n=1}^\infty$ in
a suitable (quite simple) probability space $(\Omega,\Aa,P)$ such that
$P(A_n)=x_n$ for all $n\ge 1$. 

In Section 3 we discuss the connections between the convergence of
series of independent events and the probability of the union of these
events. In particular, we give an extension of the inclusion-exclusion
principle to the case of infinitely many events. 

In Section 4 we determine a sharp lower bound for the probability of a
union of independent events when the sum of the probabilities is given
and, vice versa, a sharp upper bound for the sum of the probabilities when
the probability of the union is given. 

\section{The correspondence between sequences of independent events
  and non-negative series} 

Throughout this paper, let $(\Omega,\Aa,P)$ be a probability space,
i.e. $\Omega$ is an arbitrary non-empty set, $\Aa$ a $\sigma$-algebra
of subsets of $\Omega$ (the sets considered to be measurable
w.r.t. $P$) and $P:\Aa\abb[0;1]$ a probability measure. 

Let us first recall that infinitely many events $A_1,A_2,\ldots\in\Aa$ are {\bf
  (mutually) independent} if and only if 
$$P\kl \bigcap_{\ell=1}^{k} A_{i_\ell}\kr
=\prod_{\ell=1}^{k} P\kl A_{i_\ell}\kr \qquad
\mbox{ whenever } \; 1\le i_1<i_2<\ldots<i_k.$$ 
In the following we frequently make use of the fact that independence
isn't affected if one or several events are replaced by their
complements (with respect to $\Omega$). We denote the complement of an
event $A$ by $A^c$, i.e. $A^c:=\Omega\mi A$.  

We begin with a lemma that gives a full correspondence between series
with non-negative terms (and less than 1) and sequences of independent
events. 

\begin{lemma}\label{Corresp}
If $\{x_n\}_{n=1}^\infty$ is a sequence of real numbers $x_n\in
[0;1]$, then there exist a probability space $(\Omega,\Aa,P)$ and a
sequence $\gl A_n\gr_{n=1}^{\infty}$ of independent events $A_n\in
\Aa$ such that $P(A_n)=x_n$ for all $n$. 
\end{lemma}

\begin{proof}%[Proof/Construction]
We can choose $\Omega:=[0,1]\times [0,1]\subseteq \mathbb R^2$,
 equipped with the Lebesgue measure $P$ on the $\sigma$-algebra $\Aa$
 of Lebesgue measurable subsets of $\Omega$. We construct  the desired
 sequence of sets/events $A_n$ by recursion. First, we take $A_1$
 to be the empty set $\emptyset$ if $x_1=0,$ and if $x_1>0$ then we
 take $A_1$ to be a rectangle contained in $\Omega$, with its sides
 parallel to the axes and with area $x_1$. (Here and in the following
 it doesn't matter whether we take open or closed rectangles since
 their boundaries form a null set anyway.)

Suppose we have already defined events $A_1, A_2,\ldots,A_n$ such that
\begin{itemize}
\item[(1)] 
$P(A_k)=x_k$ for $k=1,\ldots,n$, 
\item[(2)] 
the events $A_1,\ldots,A_n$ are independent and
\item[(3)] 
each $A_k$ ($k=1,\ldots,n$) is a finite union of rectangles with sides
parallel to the axes. 
\end{itemize}
Then if if $x_{n+1}=0$ we define $A_{n+1}:= \emptyset$. If
$x_{n+1}>0$, then  for every $k\in\gl1,\ldots,n-1\gr$ and for $1\le
i_1<i_2<\ldots<i_k\le n$  
we define
$$B_{i_1,i_2,\ldots,i_k}:=\kl \bigcap_{\ell=1}^k A_{i_\ell}\kr\setminus
    \bigcup_{\substack{ 1\le j\le n\\ j\ne i_1,i_2,\ldots,i_k}}A_j,$$ 
i.e. $B_{i_1,i_2,\ldots,i_k}$ consists of those points in $\Omega$ that
    are contained in all $A_{i_\ell}$, but not in any other $A_j$. 
%\textcolor{red}{I have changed the name and meaning of $B_{i_1,i_2,\ldots,i_k}$
%  to make it more intuitive. And I think in your original proof the
%  case $k=0$ (now $k=n$) was missing, i.e. the possibility that there
%  are points that lie in all $A_j$.}
Furthermore, for $k=0$ we define
$$B_{\emptyset}:=\Omega\setminus (A_1\cup A_2\cup\ldots\cup A_n).$$
In this way, we get a decomposition of $\Omega$ to $2^n$ pairwise disjoint sets, 
$$\Omega=\bigcupdot_{ \substack{1\le k\le n\\ 1\le
    i_1<i_2<\ldots<i_k\le n}}B_{i_1,i_2,\ldots,i_k}\cup B_\emptyset.$$
For simplicity, let us re-write this as $\Omega=C_1\cupdot
C_2\cupdot\ldots\cupdot C_{2^n}$ where each $C_\ell$ is one of the sets
$B_{i_1,i_2,\ldots,i_k}$ or $B_\emptyset$. (The exact order is not
important.)

Each $C_j$ is a finite union of rectangles with sides parallel to
the axes. For each $j$ we can construct a set $\ti{C}_j\subset
C_j$ which is a union of rectangles with sides parallel to the axes
and with area $P\kl\ti{C}_j\kr=x_{n+1}\cdot P(C_j)$.  

Now we define $A_{n+1}:=\ti{C}_1\cupdot \ti{C}_2\cupdot\ldots\cupdot
\ti{C}_{2^n}$. From the construction it is obvious that $P(A_{n+1})=
x_{n+1}$ and that $A_{n+1}$ is a finite union of rectangles with
its sides parallel to the axes. It remains to show that
$A_1,\ldots,A_{n+1}$ are independent, more precisely that 
$$P\kl A_{n+1}\cap\bigcap_{\ell=1}^{k} A_{i_\ell}\kr
=P(A_{n+1})\cdot\prod_{\nu=1}^{k} P\kl A_{i_\ell}\kr \qquad
\mbox{ whenever } \; 1\le i_1<i_2<\ldots<i_k\le n.$$ 
For this purpose we fix $i_1,\ldots,i_n$ with $1\le
i_1<i_2<\ldots<i_k\le n$. Then by our construction there are
$j_1,\ldots,j_r\in\gl1,\ldots,2^n\gr$ such that
$$\bigcap_{\ell=1}^{k} A_{i_\ell}=\bigcup_{\nu=1}^{r} C_{j_\nu}.$$
Here 
$$P\kl C_{j_\nu} \cap A_{n+1}\kr
=P\kl\ti{C}_{j_\nu}\kr
=x_{n+1}\cdot P\kl C_{j_\nu}\kr
=P(A_{n+1})\cdot P\kl C_{j_\nu}\kr,$$
and we obtain
\beqaro
P\kl A_{n+1}\cap\bigcap_{\ell=1}^{k} A_{i_\ell}\kr
&=& P\kl \bigcup_{\nu=1}^{r} (C_{j_\nu} \cap A_{n+1})\kr\\
&=& \sum_{\nu=1}^{r} P\kl C_{j_\nu} \cap A_{n+1}\kr\\
&=&P(A_{n+1})\cdot \sum_{\nu=1}^{r}  P\kl C_{j_\nu}\kr\\
&=&P(A_{n+1})\cdot P\kl \bigcup_{\nu=1}^{r} C_{j_\nu}\kr\\
&=&P(A_{n+1})\cdot P\kl \bigcap_{\ell=1}^{k} A_{i_\ell}\kr
=P(A_{n+1})\cdot\prod_{\nu=1}^{k} P\kl A_{i_\ell}\kr,
\eeqaro
as desired. In such a way, we can construct the required  infinite
sequence $\{A_n\}_{n=1}^\infty.$ 
 \end{proof}

Obviously this lemma is true also for a finite number of sets.

\section{The connection between the convergence of the series of
  probabilities and the probability of the union} 

%The following theorem concerns non-negative infinite series, but it
%also has a probabilistic meaning. 

We now turn to the situation that we will deal with for the rest of
this paper. We first introduce the following notation. 

{\bf Notation.} Let $\{x_n\}_{n=1}^\infty$ be a sequence of real
numbers $x_n\in [0;1)$. Then we set 
$$T_1:=x_1,\quad T_2:=x_2(1-x_1),\quad T_3:=x_3(1-x_1)(1-x_2), $$
and generally
$$T_n:=x_n(1-x_1)(1-x_2)\cdot\ldots\cdot (1-x_{n-1})=x_n\cdot\prod_{k=1}^{n-1} (1-x_k)
\qquad\mbox{ for all } n\ge 2.$$
The quantities $T_n$ have a probabilistic meaning: In view of Lemma
\ref{Corresp}, we can consider the $x_n$ as probabilities of certain
{\it independent} events $A_n$ in some probability space:
$x_n=P(A_n)$. Then we have 
\beq\label{MeaningTn}
T_n=P\kl A_n\mi \bigcup_{k=1}^{n-1} A_k\kr,
\eeq
i.e. $T_n$ is the probability that $A_n$, but none of the events
$A_1,\ldots,A_{n-1}$ happens. In the following this correspondence will
be very useful. 

We first collect some easy observations on the $T_n$.  

\brem\label{RemTn}
\begin{enumerate}
\item[(1)]  
For all $N\in\nat$
\beq\label{SumTn}
\sum_{n=1}^N T_n=1-(1-x_1)(1-x_2)\cdot\ldots\cdot(1-x_N).
\eeq
{\bf Proof 1.} This obviously holds for $N=1$, and if it is valid for 
some $N\ge1$, then we conclude that
$$\sum_{n=1}^{N+1} T_n
=  \sum_{n=1}^{N} T_n +T_{N+1}
=1-\prod_{n=1}^{N} (1-x_n) +x_{N+1}\prod_{n=1}^{N} (1-x_n)
=1-\prod_{n=1}^{N+1} (1-x_n) ,$$
so by induction our claim holds for all $N$. 

{\bf Proof 2.} (\ref{SumTn}) also follows from the probabilistic
meaning of the $T_N$: For every $N\in\nat$ we have in view of
(\ref{MeaningTn}) 
$$\sum_{n=1}^N T_n=P\left(\bigcup_{n=1}^NA_n\right)=1-P\left(\bigcap_{n=1}^N A_n^c
\right)=1-(1-x_1)(1-x_2)\ldots(1-x_N),$$
where the last equality holds since $A_1^c,A_2^c,\ldots,A_n^c$ are also
independent events. In other words, both sides of (\ref{SumTn}) denote
the probability that (at least) one of the events $A_1,\ldots,A_N$
happens. 
\item[(2)]  
From (\ref{SumTn}) we immediately obtain 
\beq\label{Tn-Altern}
T_N=x_N\kl 1-\sum_{n=1}^{N-1} T_n\kr \qquad \mbox{ for all } N\ge1.
\eeq
\item[(3)] 
For every $N\in\nat$ the map 
$$F:\mathbb R^N\to \mathbb R^N, \;
F(x_1,x_2,\ldots,x_N):=( T_1,T_2,\ldots,T_N)$$ 
is injective (though of course not surjective) and the inverse is given by 
 $F^{-1}(T_1,T_2,\ldots,T_N)=(x_1,x_2,\ldots,x_N)$ where 
$$x_1=T_1, \ x_2=\frac{T_2}{1-T_1}, \ x_3=\frac{T_3}{1-T_1-T_2
  },\ldots,  x_N=\frac{T_N}{1-T_1-T_2-\ldots-T_{N-1}}.$$  
Thus, we will often say $\{x_n\}_{n=1}^N$ and the ``corresponding''
  $\{T_n\}_{n=1}^N$ and vice versa. The above is also true for
  $N=\infty$ in an obvious manner. 
\item[(4)]  
If $\sigma$ is some permutation of  $\gl 1,\ldots,N\gr$ and $\ti{T}=\kl
\ti{T}_1,\ldots,\ti{T}_N\kr:=F(x_{\sigma(1)},\ldots,x_{\sigma(N)})$, then
$\sum\limits_{n=1}^N \ti{T}_n=\sum\limits_{n=1}^N T_n.$ This is an
immediate consequence from (\ref{SumTn}). 
\end{enumerate}
\erem

\begin{thm}\label{thm1}
If the $T_n$ are defined as above, then 
$$\sum_{n=1}^N T_n< 1 \quad\mbox{ for all } N\in\mathbb N
\qquad \mbox{and } \qquad 
\sum_{n=1}^\infty T_n\le 1.$$ 
Furthermore $\sum\limits_{n=1}^\infty T_n=1$ if and only if
$\sum\limits_{n=1}^\infty x_n=\infty $. 
\end{thm}

\begin{proof}
$\sum\limits_{n=1}^N T_n< 1$ follows immediately from (\ref{SumTn}),
keeping in mind that $x_n<1$ for all $n$. Hence
$u:=\sum\limits_{n=1}^\infty T_n\le 1.$ 
%\textcolor{red}{Maybe it's better to skip the first proof of
%    the case of equality in order to make the part of the paper with
%    known results shorter. On the other hand, I like it to give
%    different approaches to the same result...}

If $u<1$, then we use that from (\ref{Tn-Altern}) we have
\begin{equation}\label{1}
T_n=x_n\kl 1-\sum_{k=1}^{n-1} T_k\kr\ge x_n(1-u) \qquad\mbox{ for all } n,
\end{equation}
which yields 
\begin{equation}\label{2}
\sum_{n=1}^\infty T_n\ge(1-u)\sum_{n=1}^\infty x_n, 
\qquad \mbox{ hence } \qquad
\sum_{n=1}^\infty x_n\le\frac{u}{1-u}<\infty.
\end{equation}
%Equality in \eqref{1} holds if and only if $0=T_N=T_{N+1}=T_{N+2}=\ldots,$ and then $T_n=x_n=0$ for $n\ge N.$
%Equality in \eqref{2} holds if and only if equality in \eqref{1} holds for every $N,$ i.e.,
%if and only if $0=x_1=x_2=\ldots\, ,$ i.e., if and only if $0=T_1=T_2=\ldots\, ,$  i.e., if and only if $s=0.$
Suppose now that $u=1$. We want to show that $\sum\limits_{n=1}^\infty
x_n=\infty.$ Indeed, if $\sum\limits_{n=1}^\infty x_n<\infty,$ then
there exists an $N$ such that $\sum\limits_{n=1}^\infty x_{N+n}\le
\frac{1}{2}$, and we obtain
\beqaro
\sum\limits_{n=1}^\infty T_n
&=&T_1+\ldots+ T_N+\sum_{n=1}^\infty x_{N+n}(1-T_1-T_2-\ldots-T_{N+n-1})\\ 
%&\le& T_1+\ldots+T_N+(1-T_1-T_2-\ldots-T_N)\cdot \sum_{n=1}^\infty x_{N+n}\\
&\le& T_1+\ldots+T_N+(1-T_1-T_2-\ldots-T_N)\cdot \frac{1}{2}<1
\eeqaro
since $T_1+\ldots+T_N<1$. This completes the proof of our Theorem. 

In the proof of the second statement (on the case of equality) we can also
argue as follows: Taking the limit $N\to\infty$ in (\ref{SumTn}) we obtain
$$\sum_{n=1}^\infty T_n=\lim_{N\to\infty}\sum_{n=1}^NT_n=1-\prod_{n=1}^\infty(1-x_n).$$
By the theory of infinite products \cite[p.~192]{Ahlfors}
$\sum\limits_{n=1}^\infty x_n<\infty$ is equivalent to
$\prod\limits_{n=1}^\infty(1-x_n)>0,$ hence to
$\sum\limits_{n=1}^\infty T_n<1$. 
\end{proof}

%Continuing with this line of ideas, we will now get an estimate
%for $P\left(\bigcup\limits_{n=1}^\infty A_n\right)$ in terms of
%$\sum\limits_{n=1}^\infty P(A_n)$ for independent events
%$\{A_n\}_{n=1}^\infty.$

Continuing with this line of ideas, we can get the following estimate for 
$\sum\limits_{n=1}^\infty x_n $. 

\begin{thm}\label{thm2} 
If $\sum\limits_{n=1}^\infty x_n<\infty$ and
$u:=\sum\limits_{n=1}^\infty T_n<1,$ then  
$$\sum\limits_{n=1}^\infty  x_n<\Log \frac{1}{1-u},$$ and this
estimate is sharp.  
\end{thm}

%\begin{remark}\label{rem3}
%It is obvious that $\sum\limits_{n=1}^\infty T_n<\Log \frac{1}{1-s}$ since $s<\Log \frac{1}{1-s}.$
%\end{remark}

\begin{proof}
As in the proof of Theorem \ref{thm1}, from (\ref{SumTn}) we get
$$u=\sum\limits_{n=1}^\infty T_n=1-\prod\limits_{n=1}^\infty(1-x_n).$$
Using the well-known estimate $\Log(1+x)<x$ which holds for $-1<x\le 1$ we obtain
\beq\label{EstimSumXn}
\sum_{n=1}^\infty x_n<-\sum_{n=1}^\infty \Log (1-x_n)
=-\Log \prod_{n=1}^\infty(1-x_n)
=-\Log(1-u)=\Log \frac{1}{1-u}.
\eeq
In order to show the (asymptotic) sharpness of this estimate, we fix
some $u\in [0;1)$, and we choose the $x_n$ such that finitely many
of them have the same value and all others are zero. More precisely,
for given $N\in\nat$ we set  
$$x_n:=\gl\begin{array}{ll} 
1-\sqrt[N]{1-u} & \mbox{ for } n=1,\ldots,N,\\
0 & \mbox{ for } n>N.\end{array}\r.$$
Then 
$$\sum_{n=1}^\infty T_n
=1-\prod\limits_{n=1}^\infty(1-x_n)
=1-\prod\limits_{n=1}^N \sqrt[N]{1-u}=u$$
and 
$$\sum_{n=1}^\infty x_n
=\sum_{n=1}^N x_n
=N\left(1-\sqrt[N]{1-u}\right)\underset {N\to\infty}\longrightarrow \Log \frac{1}{1-u};$$
the latter limit is easily calculated by considering the derivative of
$g(x):= (1-u)^x$ at $x=0.$  

The sharpness of the estimate can also be seen by estimating the
error in the inequality $\Log(1+x)<x$ used above: From
(\ref{EstimSumXn}) and the Taylor expansion of the logarithm we obtain
\beqaro\label{3}
0<\Log \frac{1}{1-u}-\sum\limits_{n=1}^\infty x_n 
&=&-\sum_{n=1}^\infty[\Log (1-x_n)+x_n]\\
&=&\sum_{n=1}^\infty\left(\frac{x_n^2}{2}-\frac{x_n^2}{3}+\frac{x_n^2}{4}-\frac{x_n^2}{5}+\ldots\right)
<\sum_{n=1}^\infty\frac{x_n^2}{2}.
\eeqaro
If again $x_1,\ldots,x_N$ are all equal to $x$ and $x_n=0$ for all
$n>N$ (where of course $x$ depends on $N$, in order to ensure
$\sum_{n=1}^{\infty} T_n=u$), then 
$$N\cdot x= \sum\limits_{n=1}^{\infty} x_n<\Log \frac{1}{1-u},
\qquad \mbox{ hence }  \qquad 
x^2<\frac{(\Log \frac{1}{1-u})^2}{N^2}, $$ 
and we obtain  
$$0<\Log \frac{1}{1-u}-\sum\limits_{n=1}^\infty x_n 
\le \sum_{n=1}^\infty\frac{x_n^2}{2}
=\sum_{n=1}^N \frac{x^2}{2}<\frac{\left(\Log\frac{1}{1-u}\right)^2}{2N}.$$
This upper bound obviously tends to 0 if $N\to\infty$ which again
shows the sharpness of the result. 
\end{proof}

We will revisit the estimate in Theorem \ref{thm2} from a slightly
different point of view in the next section. 

We now want to give a probabilistic formulation of Theorems \ref{thm1}
and \ref{thm2}. In order to do so we recall that if the $x_n$ are the
probabilities of certain independent events $A_n$, then $T_n$ is the
probability of $A_n\mi \bigcup_{k=1}^{n-1} A_k$. Since
these sets are pairwise disjoint, we conclude that 
$$\sum_{n=1}^{N} T_n=P\kl\bigcup_{n=1}^{N} A_n\kr.$$
So the estimate $\sum\limits_{n=1}^N T_n\le \sum\limits_{n=1}^N x_n$
(a direct consequence of $T_n\le x_n$) is just a reformulation of the
trivial inequality $P\kl\bigcup_{n=1}^{N} A_n\kr\le \sum_{n=1}^{N} P\kl
A_n\kr$. In view of (\ref{SumTn}) it is also equivalent to the estimate 
$$1- \prod\limits_{n=1}^N(1-x_n)\le \sum\limits_{n=1}^N x_n$$
valid for all $x_n\ge 0$ which of course can also be proved by
an elementary induction.

The probabilistic meaning of the sum $\sum_{n=1}^{N} T_n$ also
carries over to the limit case $N\to\infty$. To see this, let us
recall some known facts from probability theory. 

%\begin{defn}\label{defn1} \cite[p.4]{Gut}
If $\gl B_n\gr_{n\ge 1}$ is a sequence of subsets of $\Omega$, then we define
$$B_*=\liminf\limits_{n\to\infty} B_n:=\bigcup\limits_{n=1}^\infty\bigcap\limits_{k=n}^\infty B_k
\qquad \mbox{ and } \qquad 
B^*=\limsup\limits_{n\to\infty} B_n:=\bigcap\limits_{n=1}^\infty
\bigcup\limits_{k=n}^\infty B_k.$$
Obviously, we always have $B_*\subseteq B^*$. In  the case of equality
we write $\lim\limits_{n\to\infty}B_n:=B_*=B^*$. A sufficient condition
for $B_*=B^*$, hence for the existence of $\limn B_n$ is that the
sequence $\gl B_n\gr_{n\ge 1}$ is increasing 
($B_1\subseteq B_2\subseteq B_3\subseteq\ldots$) or decreasing
($B_1\supseteq B_2\supseteq B_3\supseteq\ldots$). When $\lim\limits_{n\to\infty}
B_n$ exists, then 
\beq\label{LimVertausch}
\lim_{n\to\infty} P(B_n)=P(\lim\limits_{n\to\infty} B_n)
\eeq
(see, for example \cite[p.~12]{Gut}).  
%\end{defn}

We apply this to our independent events $\{A_n\}_{n=1}^\infty.$ If
we set $B_N:=\bigcup_{n=1}^{N} A_n$, then $\lim_{N\to\infty} 
B_N=\bigcup\limits_{n=1}^\infty A_n$, hence 
$$P\left(\bigcup\limits_{n=1}^\infty A_n\right)
=P\kl\lim_{N\to\infty} B_N\kr
=\lim\limits_{N\to\infty} P(B_N)
=\sum_{n=1}^\infty T_n.$$ 

Now we can state Theorems \ref{thm1} and \ref{thm2} in terms of probability.

{\bf Theorem 1-P.} {\sl
Let $\{A_n\}_{n=1}^\infty$ be a sequence of independent events with
$P(a_n)<1$ for all $n\ge1.$ Then $P\left(\bigcup\limits_{n=1}^\infty
  A_n\right)<1$ if and only if $\sum\limits_{n=1}^\infty
P(A_n)<\infty.$ 
}

The direction ``$\Rightarrow$'' is reminiscent of the Borel-Cantelli
Lemma which can be stated as follows \cite[p.~96]{Bauer}: {\sl Let
  $\gl A_n\gr_{n=1}^\infty$ be a sequence of events.

\begin{itemize}
\setlength{\itemindent}{15pt}
\item[(BC1)] 
If $\sum_{n=1}^{\infty} P(A_n)<\infty$, then 
$$P\kl\limsupn A_n\kr=0.$$
\item[(BC2)] 
If $\sum_{n=1}^{\infty} P(A_n)=\infty$ and the events $A_n$ are
independent, then  
$$P\kl\limsupn A_n\kr=1.$$
\end{itemize}
}
Here the zero-one law due to Borel and Kolmogorov \cite[p. 47]{Bauer}
makes sure that for {\it independent} events $P\kl\limsupn A_n\kr$ has
either the value 0 or the value 1. 

In fact, (BC2) is an immediate
consequence of Theorem 1-P. Indeed, if $\sum_{n=1}^{\infty}
P(A_n)=\infty$, then also $\sum_{n=N}^{\infty} P(A_n)=\infty$ for all
$N\in\nat$, and if the events $A_n$ are independent, then Theorem 1-P
yields $P\left(\bigcup\limits_{n=N}^\infty A_n\right)=1$ for all
$N\in\nat$, so from (\ref{LimVertausch}) we deduce
$$P\kl\limsupn A_n\kr
=P\left(\bigcap_{N=1}^{\infty}\bigcup\limits_{n=N}^\infty A_n\right)
=\lim_{N\to\infty} P\left(\bigcup\limits_{n=N}^\infty A_n\right)=1.$$
For the sake of completeness we'd also like to remind the reader of
the short proof of (BC1): If $\sum_{n=1}^{\infty} P(A_n)<\infty$, then
for each given $\varepsilon>0$ there is an $N\in\nat$ such that
$\sum_{n=m}^{\infty} P(A_n)<\varepsilon$ for all $m\ge N$, hence
$P\left(\bigcup\limits_{n=m}^\infty A_n\right)<\varepsilon$ for all
$m\ge N$. Again in view of (\ref{LimVertausch}) this yields 
$$P\kl\limsupn A_n\kr
%=P\left(\bigcap_{m=1}^{\infty}\bigcup\limits_{n=m}^\infty A_n\right)
=\lim_{m\to\infty} P\left(\bigcup\limits_{n=m}^\infty A_n\right)
\le \varepsilon.$$
Since this holds for each $\varepsilon>0$, we conclude that
$P\kl\limsupn A_n\kr=0.$

{\bf Theorem 2-P.} {\sl
If $\{A_n\}_{n=1}^\infty$ is a sequence of independent events with
$P(A_n)<1$ for all $n\ge1$ and $u:=P\left(\bigcup\limits_{n=1}^\infty
  A_n\right)<1$, then
%$\sum\limits_{n=1}^\infty P(A_n)<\infty$, then 
$$\sum\limits_{n=1}^\infty P(A_n)<\Log \frac{1}{1-u}.$$
}

%Note that this trivially remains valid also of $P(A_n)=1$ for one $n$, since in
%this case $s=1$.

%If s=1 then the upper bound Ln(1/(1-u)) is infinity and this upper
%bound is really sharp because by Theorem 1 Sigma(P(A_n))  (Sigma(x_n)
%in Theorem 1) is indeed infinity. Maybe we insert it to the paper. 

Now let's consider for a moment only finitely many $x_n$, say
$x_1,\ldots,x_N$, and the corresponding $T_1,\ldots,T_N$. Expanding
(\ref{SumTn}) we obtain   
$$\sum_{n=1}^{N} T_n
=\sum_{n=1}^{N} x_n-\sum_{1\le i<j\le N} x_ix_j+\sum _{1\le i<j<k\le
  N} x_ix_jx_k+\ldots+(-1)^{N-1}x_1x_2\ldots x_N.$$
The probabilistic meaning of this identity is just the
inclusion-exclusion principle (here for the special case of
independent events): If once more we identify $x_n=P(A_n)$ where
$A_1,\ldots,A_N$ are independent events, then our identity takes the
form  
\beqar\label{InclExclPrin}
P\kl\bigcup_{n=1}^{N} A_n\kr
&=&\sum_{n=1}^{N} P(A_n)
-\sum_{1\le i<j\le N}  P(A_i\cap A_j) \\
&&+\sum _{1\le i<j<k\le N} P(A_i\cap A_j\cap  A_k)
+\ldots+(-1)^{N-1} P(A_1\cap\ldots\cap A_N). \nonumber
\eeqar
It is well-known that this identity (also in the general case of
non-independent events) gives rise to the so-called
Bonferroni inequalities (see, for example \cite{Galambos}), by
truncating it either after positive or after negative terms: 
$$\sum_{k=1}^{2r} (-1)^{k-1} S_k
\le P\kl\bigcup_{n=1}^{N} A_n\kr
\le \sum_{k=1}^{2r-1} (-1)^{k-1} S_k \qquad \mbox{ for all admissible }
r\ge 1,$$
where 
$$S_k:=\sum_{1\le j_1<j_2<\ldots<j_k\le N} P(A_{j_1}\cap\ldots\cap
A_{j_k}).$$

Our next theorem shows that the inclusion exclusion principle holds
also in the case of infinitely many independent events, i.e. that for
$N\to\infty$ all sums in (\ref{InclExclPrin}) are convergent.  

%Let us show now that when $\sum\limits_{n=1}^\infty$ $x_n<\infty,$ the
%series $\sum\limits_{n=1}^\infty T_n$ can also be summed ``by
%columns'' (see the definition of $T_n$ in the statement of
%\thmref{thm1}). That is, the inclusion-exclusion principle is
%applicable also when $N=\infty.$ 

\begin{thm}\label{thm3}
If $\{A_n\}_{n=1}^\infty$ is a sequence of independent events with
$P(A_n)<1$ for all $n$ and $\sum\limits_{n=1}^\infty P(A_n)<\infty,$
then  
$$P\left(\bigcup_{n=1}^\infty A_n\right)
=\sum_{k=1}^{\infty} (-1)^{k-1} S_k
\qquad\mbox{ where } \qquad
S_k:=\sum_{1\le j_1<j_2<\ldots<j_k} P(A_{j_1}\cap\ldots\cap A_{j_k}).$$
\end{thm}

\begin{proof}
Setting $x_n:=P(A_n)$, we can write $S_k$ as 
$$S_k=\sum\limits_{1\le j_1<j_2<\ldots<j_k} x_{j_1}x_{j_2}\ldots
x_{j_k}\, .$$

Since in any product of the form $x_{j_1}x_{j_2}\ldots x_{j_N},$ 
there is at least one $j_i$ with $j_i\ge N_,$ and $S_{N-1}$ is the
sum of all possibilities of products of $N-1$ different $x_i$'s, we
have the estimate  
\beq\label{QuotKrit}
S_N\le S_{N-1}\sum\limits_{n=N}^\infty x_n \qquad\mbox{ for all }
N\ge2.
\eeq
Now, since $\sum\limits_{n=1}^\infty x_n$ converges, for every $q\in
(0;1)$ we have $\sum\limits_{n=N}^\infty x_n< \frac{q}{2}$ for large
enough $N$, say for $N\ge N_0$. Inserting this into (\ref{QuotKrit})
yields  $S_N\le S_{N_0} \cdot\kl\tfrac{q}{2}\kr^{N-N_0}$, hence
\begin{equation}\label{5}
\lim_{N\to\infty}\frac{S_N}{q^N} =0 \qquad\mbox{ for all } q\in(0;1).
\end{equation}
(In fact, when infinitely many $x_n$'s are different from zero, then
$S_N\ne 0$ for every $N$, and we obtain even
$\lim\limits_{N\to\infty}\frac{S_N}{S_{N-1}}=0$.)  

In view of the convergence of the geometric series
$\sum\limits_{k=1}^{\infty} q^n$ this shows that the sum
$\sum\limits_{k=1}^{\infty} (-1)^{k-1} S_k$ is 
absolutely convergent. Hence, since all terms in $S_N$ have the same
(non-negative) sign, it follows that also the series obtained by
expanding all the products in the series 
$$\sum\limits_{n=1}^\infty T_n = \sum_{n=1}^{\infty}
\kl1-(1-x_1)(1-x_2)\ldots(1-x_n)\kr$$ 
is absolutely convergent, and thus in any order of summation it has
the same value. This proves our theorem. 
\end{proof}

\section{Upper and lower bounds for probabilities}

Let $N\in\mathbb N$ and $T_n$ be as above. We consider
the extremal problems to determine  
\begin{equation}\label{6} 
U_N(s):=\inf\left\{\sum_{n=1}^N T_n:\sum_{n=1}^Nx_n=s, 0\le x_1,\ldots,x_N\le 1\right\} 
\qquad \mbox{ for }  0\le s\le N 
\end{equation}
and 
$$S_N(u):=\sup\left\{\sum_{n=1}^N x_n:\sum_{n=1}^N T_n=u,0\le x_1,\ldots,x_N\le 1\right\}
\qquad \mbox{ for }  0\le u\le 1.$$ 
%\textcolor{red}{I've changed the name of the variables: $s\abb u$,
%  $t\abb s$, $R\abb U$, $Q\abb S$ in order to make them more intuitive:
%  $u/U$ stands for union, $s/S$ for sum. -- You restricted
%  considerations to $s<N$ and $u<1$, probably in view of the
%  probabilistic interpretation, but I think it's more natural to admit
%  also $s=N$ and $u=1$, and it makes things significantly easier.}
The infimum in the definition of $U_N(s)$ is in fact a minimum, since 
\beq\label{SumTn2}
T_1+\ldots+T_N=1-(1-x_1)(1-x_2)\ldots(1-x_N)
\eeq 
is a continuous function of $x_1,\ldots,x_N$ which is evaluated on the
compact set $\gl (x_1,\ldots,x_N)\in [0;1]^N: \sum_{n=1}^N x_n=s\gr$. A
similar reasoning shows that also the supremum in the definition of
$S_N(u)$ is a maximum. 

\bsat{}\label{QR-fin}
$$U_N(s)=1-\left(1-\frac{s}{N}\right)^N 
\qquad\mbox{ and } \qquad 
S_N(u)= U_N^{-1}(u)= N\cdot\kl 1-\sqrt[N]{1-u}\kr.$$
\esat

\begin{proof}
One might think of the method of Lagrange multipliers to calculate
$S_N(u)$ and $U_N(s)$, but (as sometimes in similar situations) it suffices
to apply the inequality between arithmetic and geometric means. It
shows that for all
$x_1,\ldots,x_N\in[0;1]$ with $\sum_{n=1}^Nx_n=s$ we have
$$\prod_{n=1}^{N} (1-x_n)\le \kl1- \frac{1}{N}(x_1+\ldots+x_N)\kr^N
=\kl1- \frac{s}{N}\kr^N,$$
with equality if and only if
$x_1=x_2=\ldots=x_N=\frac{s}{N}$. From this and (\ref{SumTn2}) we see 
$U_N(s)=1-\left(1-\frac{s}{N}\right)^N$. 

On the other hand, if $x_1,\ldots,x_N\in[0;1]$ satisfy $\sum_{n=1}^N
T_n=u\in [0;1]$, then by (\ref{SumTn2})
$$1-u=\prod_{n=1}^{N} (1-x_n)\le \kl1-\frac{1}{N}(x_1+\ldots+x_N)\kr^N,$$
again with equality if and only if all $x_n$ are equal, in which case
we have $x_1+\ldots+x_N=N\cdot\kl 1-\sqrt[N]{1-u}\kr$. This shows the
formula for $S_N(u)$. Obviously, $S_N= U_N^{-1}$. 
\end{proof}

\brem{} \label{Rem:QR}
\begin{itemize}
\item[(1)] 
$U_N(s)$ and $S_N(u)$ are strictly increasing functions of $s$ resp. of $u$,
while $U_N(s)$ is a decreasing and $S_N(u)$ an increasing function of
$N$.  

\begin{proof}
That $s\mapsto U_N(s)$ and $u\mapsto S_N(u)$ are increasing is trivial. 

Each $(x_1,\ldots,x_N)\in[0;1]^N$ with $\sum_{n=1}^Nx_n=s$ gives
rise to an $(x_1,\ldots,x_N,x_{N+1})\in[0;1]^{N+1}$ with 
$\sum_{n=1}^{N+1}x_n=s$ by setting $x_{N+1}:=0$, and the $T_1,\ldots,T_N$
corresponding to $(x_1,\ldots,x_N)$ and to $(x_1,\ldots,x_N,x_{N+1})$
are the same while $T_{N+1}=0$. Therefore the infimum in the definition of
$U_{N+1}(s)$ is taken over a superset of the set appearing in the
definition of $U_N(s)$, and we conclude that $U_{N+1}(s)\le U_N(s)$ for $N\ge
s$. A similar resoning shows that $N\mapsto S_N(u)$ is increasing. 

Of course, the monotonicity of $N\mapsto U_N(s)$ and $N\mapsto S_N(u)$
can also be verified by calculating the derivatives of
the functions $g(x):=x\Log\kl1-\frac{s}{x}\kr$ and $h(y):=y\cdot\kl
1-(1-u)^{1/y}\kr$ and showing that they are non-negative. 
%Differentiating with respect to $N$ gives (we   treat $N$ as a continuous variable)
%$$\frac{dU_N(t)}{dN}=-\left(1-\frac{t}{N}\right)^N\left\{\Log \left(1-\frac{t}{N}\right)+\frac{t/N}{1-t/N}
%\right\}.$$
%Set $h(x)=\Log (1-x)+\frac{x}{1-x},$ $ 0\le x<1.$ Then
%$h'(x)=\frac{x}{(1-x)^2}>0$ for $0<x<1,$ $h(0)=0.$ Thus, consider
%$\frac{t}{N}$ as $x.$ We get that $\frac{d{U_N}(t)}{d_N}<0,$ and the
%claim is proved. 
\end{proof}
\item[(2)] 
In view of (1), the maximum of $U_N(s)$ over all $N\ge s $ is attained
at the first one, i.e., at $N=\lceil s\rceil$ (where $\lceil s\rceil$
denotes the smallest integer $\ge s$). Hence we have 
$$U(s):=\max_{N\ge s}U_N(s)=1-\left(1-\frac{s}{\lceil
    s\rceil}\right)^{\lceil s\rceil}.$$ 
%Set for the moment $n=[t],$ then for $n\le t<n+1$, we have
%$R(t)=1-\left(1-\frac{t}{n+1}\right)^{n+1},$ an increasing function of
%$t.$ 

So when the finite number of events is $N=\lceil s\rceil,$ the minimum
of the probabilities $P\left(\bigcup\limits_{n=1}^N A_n\right)$ (under
the restriction $\sum_{n=1}^{N} P(A_n)=s$) is the
highest. Also, we have $\lim\limits_{s\to\infty} U(s)=1.$ 
%Similarly, 
%$$Q(u):=\min\limits_{N\in\nat} S_N(u)=S_1(u).$$
\end{itemize}
\erem

%Is it possible that when $N=\infty,$ the infimum/minimum of the
%probability of the union is less or equal to
%$1-\left(1-\frac{t}{[t]+1}\right)^{[t]+1}?$ 

%By the following claim the answer is in the negative.

In an obvious way, we can extend the definitions of $S_N$ and $U_N$
also to the case $N=\infty$. We will show that we will obtain explicit
formulas for $S_\infty$ and $U_\infty$ by taking the limits of $S_N$
and $U_N$ for $N\to\infty$. 

First of all we note that for $s>0$ the infimum in the definition of
$U_\infty(s)$ is not a minimum. Indeed, suppose that $\sum\limits_{n=1}^\infty x_n=s$ and
$1-\prod\limits_{n=1}^\infty(1-x_n)=U_\infty(s).$ W.l.o.g. we can assume that
$x_1>0$. Then we replace $x_1$
by $\frac{x_1}{2},$ $\frac{x_1}{2},$ i.e. we create a new 
sequence $\{x_n'\}_{n=1}^\infty$ where $x_1'=x_2'=\frac{x_1}{2}$ and
$x_n'=x_{n-1}$ for $n\ge3.$ Then $\sum\limits_{n=1}^\infty x_n'=s$ and
$(1-x_1')(1-x_2')=\left(1-\frac{x_1}{2}\right)^2>1-x_1.$  Hence
$$1-\prod\limits_{n=1}^\infty(1-x_n')<1-\prod_{n=1}^\infty(1-x_n)=U_\infty(s),$$
and we get a contradiction. 

\bsat \label{QR-infin}
$$U_\infty(s)=1-e^{-s} %\quad\mbox{ for } s\ge 0
\qquad \mbox{ and } \qquad 
S_\infty(u)=(U_\infty)^{-1}(u)=\Log\frac{1}{1-u}. $$ %\quad\mbox{ for } 0\le u\le 1.
\esat

This formula for $S_\infty(u)$ gives also a new proof of Theorem
\ref{thm2}. 

\begin{proof}
Since the infimum in the definition of $U_{\infty}(s)$ is taken over a
larger set than for any $U_N(s)$ (cf. the proof of Remark \ref{Rem:QR}
(1)) and since $U_N(s)\underset {N\to\infty}\searrow 1-e^{-s},$ it is
clear that $U_\infty(s)\le 1-e^{-s}.$ 

Suppose that $U_\infty(s)<1-e^{-s}$ for some $s\ge 0$. Then for some
sequence $\{x_n\}_{n=1}^\infty$ with $\sum\limits_{n=1}^\infty x_n=s$
we have $u_0:=\sum\limits_{n=1}^\infty T_n<1-e^{-s}$. Here infinitely
many $x_n$ are positive since otherwise $U_N(s)$ would be a lower
bound for $\sum\limits_{n=1}^\infty T_n$ for sufficiently large $N$,
contradicting $U_N(s)\ge 1-e^{-s}>u_0$. 
.
We can choose $N_0$ so large that $N_0>s$ and $u_0<1-e^{-s_0}$ where
$s_0:=\sum\limits_{n=1}^{N_0}x_n<s$. We then have 
$$\sum_{n=1}^{N_0}T_n
\le\sum_{n=1}^\infty T_n
=u_0<1-e^{-s_0}<1-\left(1-\frac{s_0}{N_0}\right)^{N_0}=U_{N_0}(s_0).$$ 
This is a contradiction to the definition of $U_{N_0}(s_0)$. Hence
$U_\infty(s)=1-e^{-s}$. 

As we have mentioned already in the proof of Theorem \ref{thm2}, 
$S_N(u)=N\cdot\kl 1-\sqrt[N]{1-u}\kr\underset {N\to\infty} \nearrow
\Log\frac{1}{1-u}$, so a similar reasoning as for $U_\infty(s)$ shows
that $S_\infty(u)=\Log\frac{1}{1-u}=(U_\infty)^{-1}(u)$. 
\end{proof}

The functions $S_N,$ $U_N,$ $S_\infty,$ $U_\infty$ are plotted in
Figure \ref{fig:US}.

\begin{figure}[htb]
  \begin{center}
    \includegraphics[width=0.475\textwidth]{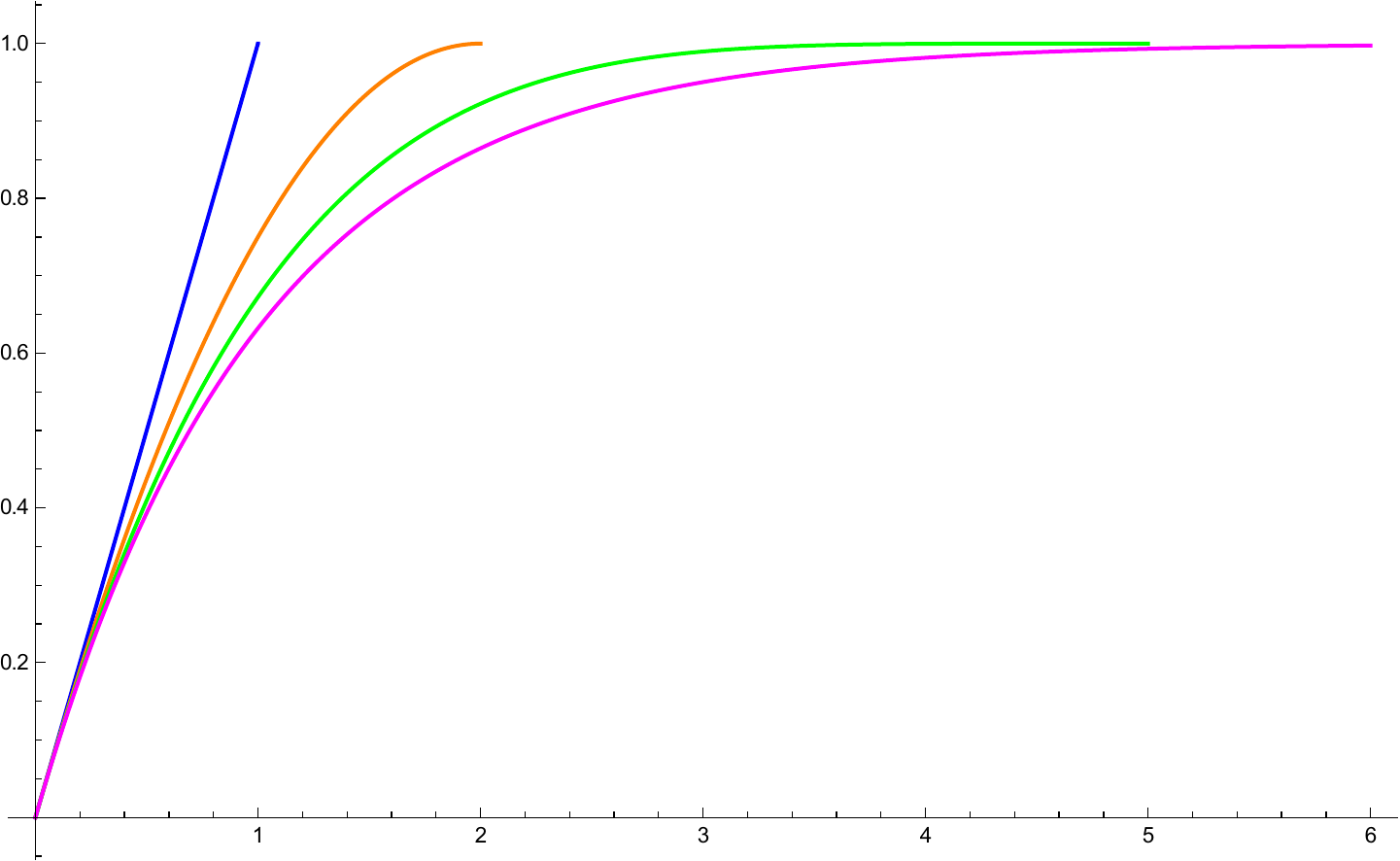}\hfill
    \includegraphics[width=0.475\textwidth]{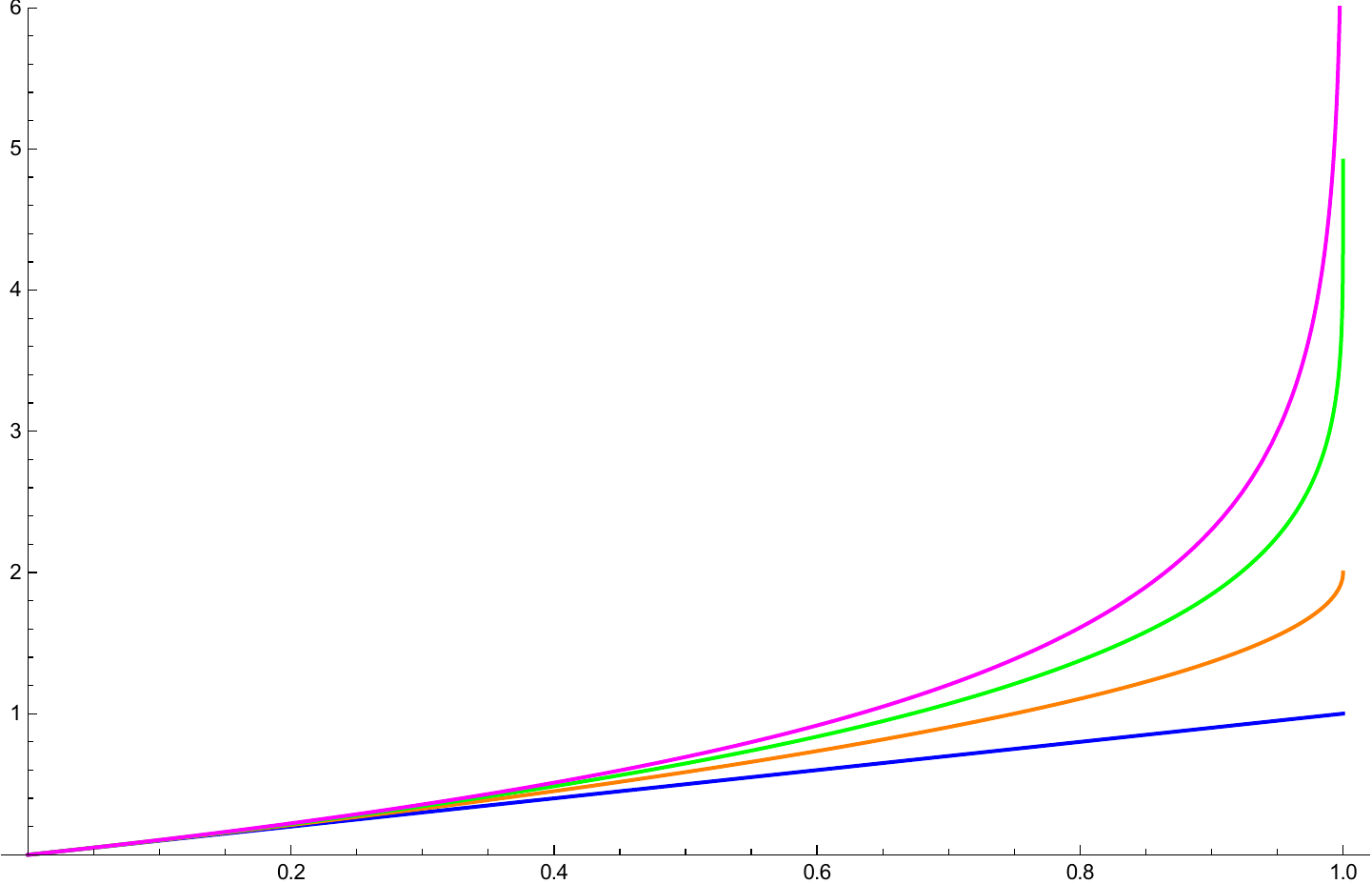}
    \caption{{\sf The graphs of $U_N(s)$ (left) and $S_N(u)$ (right)
        for $N=1,2,5,\infty$}}  
\label{fig:US}    
  \end{center}
\end{figure}

The results in Theorems \ref{QR-fin} and \ref{QR-infin} can be
reformulated in terms of probabilities of independent events: 

\bsat{} \label{mainresult}
Let $A_1,\ldots,A_N$ be finitely many independent events. Then 
\beqar \label{MainRes1fin}
P\kl\bigcup_{n=1}^{N} A_n\kr
&\ge& 1-\kl\frac{1}{N}\sum_{n=1}^{N} P(A_n)\kr^N,\\
\sum_{n=1}^{N} P(A_n) \label{MainRes2fin}
&\le& N\cdot\kl1-\sqrt[N]{1-P\kl\bigcup_{n=1}^{N} A_n\kr}\kr.
\eeqar
If $\gl A_n\gr_{n=1}^\infty$ is a sequence of independent
events such that $0<\sum_{n=1}^{\infty} P(A_n)<\infty$, then 
\beqaro
P\kl\bigcup_{n=1}^{\infty} A_n\kr
&>&1-\exp\kl-\sum_{n=1}^{\infty} P(A_n)\kr, \\
\sum_{n=1}^{\infty} P(A_n)
&<& \Log\frac{1}{1-P\kl\bigcup_{n=1}^{\infty}  A_n\kr}.
\eeqaro
All these estimates are best-possible. Equality in (\ref{MainRes1fin})
and (\ref{MainRes2fin}) occurs if $P(A_1)=\ldots=P(A_N)$.
\esat

The assumption in Theorem \ref{mainresult} that the events $A_n$ are
independent is essential as the following easy counterexample
demonstrates: Choose $A_1=\ldots=A_N$ to be one and the same event,
whose probability is $x=P(A_1)\in (0;1)$. Then the left hand side of
(\ref{MainRes1fin}) is $x$ while the right hand side is $1-x^N$ which
will be larger than $x$ if $N$ is sufficiently large. Similarly, the
left hand side of (\ref{MainRes2fin}) is $Nx$ while its right hand
side is $N(1-\sqrt[N]{1-x})$, so their quotient
$\tfrac{1-\sqrt[N]{1-x}}{x}$ will become arbitrarily small for
sufficiently large $N$.

%\section{Upper and lower bounds}

At last we take a brief look at the extremal problems opposite to those
above, i.e. with supremum replaced by infimum and vice versa. %More
%precisely, for $N\in\nat\cup\gl\infty\gr$ we want to determine 
%$$\inf\left\{\sum_{n=1}^N x_n:\sum_{n=1}^N T_n=u,0\le x_1,\ldots,x_N\le 1\right\}
%\qquad \mbox{ for }  0\le u\le 1$$ 
%and
%$$\sup\left\{\sum_{n=1}^N T_n:\sum_{n=1}^Nx_n=s, 0\le x_1,\ldots,x_N\le 1\right\} 
%\qquad \mbox{ for }  0\le s\le N. $$
Their solutions turn out to be quite simple. 

\bsat
For all $N\in\nat$ we have
$$\inf\left\{\sum_{n=1}^N x_n:\sum_{n=1}^N T_n=u,0\le x_1,\ldots,x_N\le  1\right\}
=u \qquad \mbox{ for }  0\le u\le 1,$$ 
$$\sup\left\{\sum_{n=1}^N T_n:\sum_{n=1}^Nx_n=s, 0\le x_1,\ldots,x_N\le  1\right\}
 =\min\gl s,1\gr \qquad \mbox{ for }  0\le s\le N,$$
and this remains valid analogously also for $N=\infty$.
\esat

\begin{proof} 
By the definition of the $T_n$ we always have $x_n\ge T_n,$
hence $\sum_{n=1}^N T_n\le \sum_{n=1}^N x_n$. Therefore the infimum is
at least $u,$ and the value $u$ is attained by taking only one event
with $x_1=T_1=u$ (and $x_n=0$ for $n\ge 2$). This shows the first
assertion.  
%When the number of events is $N\ge 2$ (and we of course
%  also ask that each $T_n,$ $x_n$ be positive), then 
%$$x_1+x_2+\ldots+x_N=T_1+\frac{T_2}{1-T_1}+\ldots+\frac{T_N}{1-T_1-\ldots-T_{N-1}}.$$
%Then we can get values as close as we like to $s$ (but bigger than
%$s$). Hence the infimum is $s$ but is not attained. 

The very same reasoning applies to the case $s\le 1$ in the second
assertion. If $s>1,$ we can choose   $x_1=1$ and the other $x_n$ more
or less arbitrary, requiring only $\sum_{n=1}^Nx_n=s$. Then $T_1=1$ and
$T_n=0$ for all $n\ge 2$, hence $\sum_{n=1}^N T_n=1$ which is of
course the maximal value. This proves also the second assertion.
\end{proof}

{\bf Acknowledgment.} We would like to thank Professor Ely Merzbach
for his valuable advice.

\vspace{10pt}
\parbox{85mm}{\sl J\"urgen Grahl\\
University of W\"urzburg \\
Department of Mathematics  \\     
97074 W\"urzburg\\
Germany\\
e-mail: grahl@mathematik.uni-wuerzburg.de}
\hfill
\parbox{85mm}{\sl Shahar Nevo \\
Bar-Ilan University\\
Department of Mathematics\\
Ramat-Gan 52900\\
Israel\\
e-mail: nevosh@math.biu.ac.il}
\end{document}